\documentclass{amsart}

\newtheorem{theorem}{Theorem}[section]
\newtheorem{lemma}[theorem]{Lemma}
\newtheorem{proposition}[theorem]{Proposition}
\newtheorem{corollary}[theorem]{Corollary}

\theoremstyle{definition}
\newtheorem{definition}[theorem]{Definition}

\theoremstyle{remark}
\newtheorem{remark}[theorem]{Remark}

\numberwithin{equation}{section}

\begin{document}

\title[Spaceability and algebrability of nowhere integrable functions ]{Spaceability and algebrability of sets of nowhere integrable functions}

\author[ Sz. G\l \c ab]{ Szymon G\l \c ab}
\address{Institute of Mathematics, Technical University of
\L\'od\'z, W\'olcza\'nska 215, 93-005 \L\'od\'z,   Poland}
\email{szymon.glab@p.lodz.pl}

\author[P. L. Kaufmann]{ Pedro L. Kaufmann }
\address{Instituto de matem\'atica e estat\'istica, Universidade de S\~ao Paulo, Rua do Mat\~ao, 1010, CEP 05508-900, S\~ao Paulo,  Brazil}
\email{plkaufmann@gmail.com}
\thanks{The second author was supported by CAPES, Research Grant PNPD 2256-2009.}

\author[L. Pellegrini]{ Leonardo Pellegrini}
\address{Instituto de matem\'atica e estat\'istica, Universidade de S\~ao Paulo, Rua do Mat\~ao, 1010, CEP 05508-900, S\~ao Paulo,  Brazil}
\email{leonardo@ime.usp.br}



\commby{Thomas Schlumprecht}

\begin{abstract}

We show that the set of Lebesgue integrable functions in $[0,1]$ which are nowhere essentially bounded is spaceable, improving a result from \cite{gms}, and that it is strongly $\mathfrak{c}$-algebrable. We prove strong $\mathfrak{c}$-algebrability and non-separable spaceability of the set of functions of bounded variation which have a dense set of jump discontinuities. Applications to sets of Lebesgue-nowhere-Riemann integrable and Riemann-nowhere-Newton integrable functions are presented as corollaries. In addition we prove that the set of Kurzweil integrable functions which are not Lebesgue integrable is spaceable (in the Alexievicz norm) but not $1$-algebrable. We also show that there exists an infinite dimensional vector space $S$ of differentiable functions such that each element of the $C([0,1])$-closure of $S$ is a primitive to a Kurzweil integrable function, in connection to a classic spaceability result from \cite{gurrus}. 

\end{abstract}

\maketitle

\section{Introduction and terminology}

This work is a contribution to the study of large linear and algebraic structures within essentially nonlinear sets of functions which satisfy special properties; the presence of such structures is often described using the terminology \emph{lineable}, \emph{algebrable} and \emph{spaceable}.  Recall that a subset $S$ of a topological vector space $X$ is said to be \emph{lineable} (respectively, \emph{spaceable}) if $S\cup\{0\}$ contains an infinite dimensional vector subspace (respectively, a \emph{closed} infinite dimensional  vector subspace) of $X$; this terminology was first introduced in \cite{eg} (see also \cite{ags}). The term \emph{algebrability} was introduced later in \cite{as}; if $X$ is a linear algebra, $S$ is said to be $\kappa$-algebrable if $S\cup\{0\}$ contains an infinitely generated algebra, with a \emph{minimal} set of generators of cardinality $\kappa$ (see \cite{as} for details). We shall work with a strenghtened notion of $\kappa$-algebrability, namely, \emph{strong $\kappa$-algebrability}. The definition follows:
\begin{definition}
We say that a subset $S$ of an algebra $\mathcal{A}$ is \emph{strongly $\kappa$-algebrable}, where $\kappa$ is a cardinal number, if there exists a $\kappa$-generated free algebra $\mathcal{B}$ contained in $S\cup\{0\}$. 
\end{definition}

We recall that, for a  cardinal number $\kappa$, to say that an algebra $\mathcal{A}$ is a \emph{$\kappa$-generated free algebra}, means that there exists a subset $Z=\{z_\alpha:\alpha<\kappa\}\subset \mathcal{A}$ such that any function $f$ from $Z$ into some algebra $\mathcal{A'}$ can be uniquely extended to a homomorphism from $\mathcal{A}$ into $\mathcal{A'}$. The set $Z$ is called a \emph{set of free generators} of the algebra $\mathcal{A}$. If $Z$ is a set of free generators of some subalgebra $\mathcal{B} \subset \mathcal{A}$, we say that $Z$ is a set of free generators \emph{in} the subalgebra $\mathcal{A}$. If $\mathcal{A}$ is \emph{commutative}, a subset $Z=\{z_\alpha:\alpha<\kappa\} \subset \mathcal{A}$ is a set of free generators in $\mathcal{A}$ if for each polynomial $P$ and for any $z_{\alpha_1},z_{\alpha_2},\dots,z_{\alpha_n} \in Z$ we have 
$$
P(z_{\alpha_1},z_{\alpha_2},\dots,z_{\alpha_n})=0 \mbox{ if and only if } P=0. 
$$
The definition of strong $\kappa$-algebrability was introduced in \cite{bsz}, though in several papers, sets which are shown to be algebrable are in fact strongly algebrable, and that is seen clearly by the proofs. See \cite{as} and \cite{pathos}, among others. 
Strong algebrability is in effect a stronger condition than algebrability: for example, $c_{00}$ is $\omega$-algebrable in $c_0$ but it is not strongly $1$-algebrable (see \cite{bsz}).  \\

All functions in this paper are real valued and defined in $[0,1]$, unless stated otherwise. We are particularly interested in functions which are integrable with respect to some definition of integral, but \emph{nowhere integrable} with respect to a different definition of integral. We clarify that the notation \emph{$f$ is nowhere T-integrable}, for some integration process T, means that $f$ is not T-integrable in any subinterval of the domain. 

In section \ref{seclebnriem}, we show that the set $\mathcal{G}$ of nowhere essentially bounded (not essentially bounded in any subinterval of the domain) Lebesgue integrable functions is spaceable and strongly $\mathfrak{c}$-algebrable. The fact that $\mathcal{G}$ is spaceable improves a result from Garc\'{i}a-Pacheco, Mart\'{i}n, and Seoane-Sep\'ulveda \cite{gms}, which states that the set of Lebesgue integrable functions which are not equivalent to Riemann integrable functions is spaceable. The improvement is seen clearly in Corollary \ref{corolGspaceable}. In section \ref{secriem}, we investigate spaceability and algebrability properties of the set of functions of bounded variation which have a dense set of jump discontinuities - which are, in turn, all Riemann-nowhere-Newton integrable. A connection with a result from \cite{pathos} will be pointed out. 
Section \ref{seckurnleb} is dedicated to Kurzweil integrable functions. We show that the set $\mathcal{J}$ of Kurzweil integrable functions which are not Lebesgue integrable is spaceable, but $\mathcal{J}$ does not contain any nontrivial algebra. We discuss the relation with a result from Gurariy \cite{gurrus} (see also \cite{gur}), which states that the set of differentiable functions is lineable but not spaceable in $C([0,1])$.
At the end of sections \ref{seclebnriem} and \ref{seckurnleb}, we include some problems left open and remarks which stimulate further investigation.  \\

%%%%%%%%%%%%%%%%%%%%%%%%%%%%%%%%%%%%%%%%%%%%%%%%%%%%%%%%%%%%%%%%%%%%%

\section{Spaceability and algebrability of sets of Lebesgue integrable functions}
\label{seclebnriem}

Let $\mathcal{G}$ be the set of all Lebesgue integrable functions in $[0,1]$ which are nowhere essentially bounded. The main results is this Section are that $\mathcal{G}$ is spaceable in the $L^1$ norm and that it is strongly $\mathfrak{c}$-algebrable. Our construction of functions in $\mathcal{G}$ involves infinite unions of Cantor sets. We recall that a Cantor set $A$ in $[a,b]$ is a perfect, not countable, nowhere dense subset of $[a,b]$ of diameter $b-a$ which is obtained by subtracting from $[a,b]$ a countable union of open sets in a special way, so that $A$ can measure from zero to strictly less than $b-a$. We shall call these open sets the \emph{holes} of $A$, and use also the following notation:

\begin{definition}

We say that a nonvoid subset $B$ of $[a,b]$ of the form $B=\cup\{A_j:j\in\Gamma\}$ is \emph{Cantor-built (with Cantor components $A_j$)} if  
\begin{enumerate}
\item each $A_j$ is a Cantor set in some subinterval $[c_j,d_j]$ of $[a,b]$; 
\item if $j\neq k$, then $m([c_j,d_j]\cap [c_k,d_k])=0$, where $m$ denotes Lebesgue measure. 
\end{enumerate}

\end{definition} 

\begin{remark} Clearly a Cantor-built set has an at most countable amount of Cantor components. The index set $\Gamma$ is therefore nonvoid and at most countable. 
\end{remark}

For each subset $A \subset [0,1]$ and each subinterval $I=[a,b]\subset [0,1]$, let us denote 
$$A_I \doteq \left\{(b-a)x+a:x\in A\right\}.$$ 
Note that, if $A$ is measurable, then $A_I$ is measurable and $m(A_I)=(b-a)m(A)$.\\

We shall now define a family of functions $f_j$ in $\mathcal{G}$. First we define a family of pairwise disjoint subsets $A_j$ of $[0,1]$ as follows. 
Let $A_1$ be a Cantor set in $[0,1]$ of measure $2^{-1}$. Define 
$$A_2\doteq \cup \{(A_1)_I: I\mbox{ is a hole of } A_1\}.$$
Then $A_2$ is Cantor-built, with Cantor components $(A_1)_I$ like above, and $m(A_2)=2^{-2}$. 
Let $j\geq 3$, suppose that $A_j\subset [0,1]$ is Cantor-built, and that $m(A_j)=2^{-j}$. Define 
$$A_{j+1}\doteq \cup \{(A_1)_I: I\mbox{ is a hole of a Cantor component of } A_j\}.$$
Then $A_{j+1}$ is Cantor-built, with Cantor components $(A_1)_I$,  and $m(A_{j+1})=2^{-(j+1)}$. 
%$A_2$ was defined expicitly only for didactic purposes; the definition would coincide with the recursive one. 
We have obtained a sequence $(A_j)$ of Cantor-built subsets of $[0,1]$ satisfying 
\begin{enumerate}
\item $m(A_j\cap A_k) = 0$, when $j\neq k$; 
\item $m(\cup_j A_j) = 1$ (thus $\cup_j A_j$ is dense in $[0,1]$);
\item for each subinterval $J\subset [0,1]$, there is a natural number $j_0$ such that $A_j$ has a Cantor component contained in $J$ for each $j\geq j_0$.
\end{enumerate}
Let $1<\theta <2$, write $\Theta \doteq \sum_j (\theta /2)^j$ and define, for each $j$, $f_j\doteq \theta^j \chi_{A_j}$. 
Suppose that, for each natural number $k$, $(n_j^k)_j$ is a strictly increasing sequence of natural numbers, and that those sequences are pairwise disjoint. For each $k$ define 
$$g_k \doteq \sum_j f_{n_j^k}.$$
Note that each $g_k\in \mathcal{G}$, and 
$$0 < \int g_k = \sum_j \int f_{n_j^k} = \sum_j \left(\frac{\theta}{2}\right)^{n_j^k} <\Theta;$$
moreover, $\{g_k\}_k$ is linearly independent in $L^1$, since the sets $\{x\in [0,1]: g_k(x) \neq 0\}$ all have positive measure and $m(\{x\in [0,1]: g_k(x) \neq 0\} \cap \{x\in [0,1]: g_j(x) \neq 0\})$ is zero whenever $j\neq k$.

\begin{theorem}

$\overline{span(\{g_k\}_k)} \subset \mathcal{G} \cup\{0\}$. In particular, $\mathcal{G}$ is spaceable in $L^1$. 

\label{Gspaceable}
\end{theorem}

\begin{proof} For each natural numbers $m_1 < m_2$ and each $a_1,\dots,a_{m_2}\in \mathbb{R}$ the inequality $\|\sum_{k=1}^{m_1}a_k g_k\|\leq \|\sum_{k=1}^{m_2}a_k g_k\|$ holds, thus $(g_k)_k$ is a basic sequence, and therefore a Schauder basis to $\overline{span(\{g_k\}_k)}$. It follows that, for a given nonzero $f\in \overline{span(\{g_k\}_k)}$, there is a nonzero sequence $(\alpha_k)_k$ of real numbers satisfying
\begin{eqnarray}
f = \sum_k \alpha_k g_k.
\label{deff}
\end{eqnarray}
Suppose  $\alpha_{k_0}\neq 0$, and let $J\subset [0,1]$ be an open interval. There exists a natural number $j_0$ such that there is a Cantor component $B_j$ of $A_{n_{j}^{k_0}}$ contained in $J$, for each $j\geq j_0$. Since each $B_j$ has positive measure, $f|_{B_j}\equiv \theta^{n_j^{k_0}}$ and $\{\theta^{n_j^{k_0}}:j\in \mathbb{N}  \}$ is unbounded, it follows that $f$ is not essentially bounded in $J$.
\end{proof}

Since Riemann integrable functions are bounded, we have the following: 

\begin{corollary}

The set of Lebesgue integrable functions in $[0,1]$ which are not equivalent to a Riemann integrable function in any subinterval of the domain is spaceable. 

\label{corolGspaceable}
\end{corollary}

We will now study the strong algebrability aspect of the set $\mathcal{G}$. First, note that the construction of the sequence of Cantor-Built sets $A_j$ is associated to the convergent series $\sum_j 2^{-j}$. It is not hard to see that with some technical effort we can build a similar sequence of Cantor-built sets $B_j$ in $[0,1]$ associated to the series $\sum_j \frac{1}{e} \frac{1}{j!}$, satisfying $m(B_j) = \frac{1}{e} \frac{1}{j!}$ and the $A_j$-like properties:

\begin{enumerate}
\item $m(B_j\cap B_k) = 0$, when $j\neq k$; 
\item $m(\cup_j B_j) = 1$;
\item for each subinterval $J\subset [0,1]$, there is a natural number $j_0$ such that $B_j$ has a Cantor component contained in $J$ for each $j\geq j_0$.

\end{enumerate}

Let $\{\theta_\alpha:\alpha<\mathfrak{c}\}$ a set of real numbers strictly greater than $1$ such that the set $\{\ln(\theta_\alpha):\alpha<\mathfrak{c}\}$ is linearly independent over the rational numbers. For each $\alpha<\mathfrak{c}$, define 
$$
g_\alpha=\sum_{j=1}^\infty\theta^j_\alpha\chi_{B_j}.
$$
For each $\alpha$ the series $\sum_j \frac{\theta_\alpha^j}{j!}$ converges, thus each $g_\alpha$ is Lebesgue integrable.

\begin{theorem}

$\{g_\alpha:\alpha<\mathfrak{c}\}$ is a set of free generators, and the algebra generated by this set is contained in $\mathcal{G}\cup\{0\}$. In particular, $\mathcal{G}$ is strongly $\mathfrak{c}$-algebrable.  

\label{lnebalg}
\end{theorem}

\begin{proof} It suffices to show that, for every $m$ and $n$ positive integers, for every matrix $(k_{ij}:i=1,\dots, m,\, j=1,\dots, n)$ of non-negative integers with non-zero and distinct rows, for every $\alpha_1,\dots,\alpha_n <\mathfrak{c}$ and for every $\beta_1,\dots,\beta_m\in\mathbb{R} $ which do not vanish simultaneously, the function
$$
g=\beta_1g_{\alpha_1}^{k_{11}}\dots g_{\alpha_n}^{k_{1n}}+\dots+ \beta_mg_{\alpha_1}^{k_{m1}}\dots g_{\alpha_n}^{k_{mn}}
$$
is in $\mathcal{G}$. $g$ is in $L^1$ since  
\begin{align*}
\int |g| & \leq \int \left[\sum_{j=1}^\infty(|\beta_1|(\theta_{\alpha_1}^{k_{11}}\cdots \theta_{\alpha_n}^{k_{1n}})^j+...+ |\beta_m|(\theta_{\alpha_1}^{k_{m1}}\cdots \theta_{\alpha_n}^{k_{mn}})^j)\chi_{B_j}\right]\\
& = \frac1e \sum_{j=1}^\infty\frac{|\beta_1|(\theta_{\alpha_1}^{k_{11}}\cdots \theta_{\alpha_n}^{k_{1n}})^j+...+ |\beta_m|(\theta_{\alpha_1}^{k_{m1}}\cdots \theta_{\alpha_n}^{k_{mn}})^j}{j!} <\infty.
\end{align*}

Since $\ln(\theta_{\alpha_1}^{k_{i1}}\cdots \theta_{\alpha_n}^{k_{in}})=
k_{i1}\ln\theta_{\alpha_1}+...+k_{in}\ln\theta_{\alpha_n}$
and $\ln\theta_{\alpha_1},...,\ln\theta_{\alpha_n}$ are $\mathbb{Q}$-linearly independent, the numbers $\ln(\theta_{\alpha_1}^{k_{i1}}\cdots \theta_{\alpha_n}^{k_{in}})$, $i=1,...,m$, are distinct. Then by the strict monotonicity of the logarithmic function we may assume that 
\begin{eqnarray}
\theta_{\alpha_1}^{k_{11}}\cdots \theta_{\alpha_n}^{k_{1n}}>\theta_{\alpha_1}^{k_{21}}\cdots \theta_{\alpha_n}^{k_{2n}}>...>\theta_{\alpha_1}^{k_{m1}}\cdots \theta_{\alpha_n}^{k_{mn}};
\label{thetas}
\end{eqnarray}
we also may assume $\beta_1\neq 0$. To simplify the notation put $\theta_i=\theta_{\alpha_1}^{k_{i1}}\cdots \theta_{\alpha_n}^{k_{in}}$. Then we can write
$$
g=\sum_{j=1}^\infty(\beta_1\theta^j_1+...+\beta_m\theta^j_m)\chi_{B_j}.
$$
From (\ref{thetas}) and since $\beta_1$ is assumed to be nonzero, we can find $j_0\in\mathbb{N}$ such that 
$$
|\beta_2|\theta^j_2+...+|\beta_m|\theta^j_m<\frac12 |\beta_1|\theta^j_1
$$
for all $j\geq j_0$. Then for those $j$
\begin{align*}
|\beta_1\theta^j_1+...+\beta_m\theta^j_m| &\geq |\beta_1|\theta^j_1 - \left| \beta_2 \theta^j_2 +...+ \beta_m\theta^j_m\right|\\
& \geq |\beta_1|\theta^j_1 - \left( |\beta_2| \theta^j_2 +...+ |\beta_m| \theta^j_m\right)\\
& >  \frac12 |\beta_1|\theta^j_1. 
\end{align*}
Since each nonvoid open subset of $[0,1]$ intercepts all $B_j$ in a non-null set (for high enough $j$), the inequality above shows that $g$ is nowhere essentially bounded. \end{proof}

In parallel to Corollary \ref{corolGspaceable}, we have the following:

\begin{corollary}

The set of Lebesgue integrable functions which are not equivalent to a Riemann integrable function in any subinterval of the domain is strongly $\mathfrak{c}$-algebrable.  

\end{corollary}

\subsection{Further questions on the algebrability of $\mathcal{G}$} 

It is unknown to the authors wether $\mathcal{G} \cup\{0\}$ admits a \emph{closed} algebra within it. We do not even know wether the free algebra constructed in the proof of Proposition \ref{lnebalg} has its closure contained in $\mathcal{G} \cup\{0\}$; a positive answer would imply automatically on the spaceability of $\mathcal{G}$. 

By another point of view, note the following:

\begin{proposition}
The set of nowhere essentially bounded functions is co-meager in $L^1$. 
\end{proposition}
 
\begin{proof} Let $N\in \mathbb{N} $ and $I$ be any subinterval of $[0,1]$ with rational endpoints. Let $M=\{f\in L^1:|f(x)|\leq N $ for a.e. $x\in I\}$. To end the proof it suffices to show that $M$ is nowhere dense. 

Let $B(f,R)$ be a ball in $L^1$ centered in some $f$ with a radius $R>0$, and take a subinterval $J$ of $I$ with length $\frac{R}{6N}$. Define $g:[0,1]\to \mathbb{R} $ by
$$
g(x)=\left\{ \begin{array}{ll}
2N, & \mbox{if } x\in J,\\
f(x), & \mbox{if } x\in[0,1]\setminus J.
\end{array}\right.
$$
Then 
$$
\|g-f\|=\int_{J}\|2N-f\|\leq m(J)3N=\frac{R}{6N}\cdot 3N=\frac{R}{2}. 
$$     
Consider the ball $B(g,R/7)$, and let $h\in B(g,R/7)$. Then $h\in M$ would imply 
$$
R/7\geq\|h-g\|=\int_0^1|h-g|\geq\int_{J}|h-g|\geq m(J)N=R/6,
$$
a contradiction. \end{proof}

In particular, $\mathcal{G}$ is dense in $L^1$. This motivates the search of algebras within $\mathcal{G}\cup\{0\}$ which are dense in $L^1$. It is not of the knowledge of the authors wether these algebras exist. Note that the algebra constructed in the proof of Proposition \ref{lnebalg} cannot be dense, since functions contained in the closure are all constant when restricted to each $B_j$.

%%%%%%%%%%%%%%%%%%%%%%%%%%%%%%%%%%%%%%%%%%%%%%%%%%%%%%%%%%%%%%%%%%%%%5

\section{Spaceability and algebrability of sets of functions of bounded variation}
\label{secriem}

In Theorem 4.1 from \cite{pathos}, Garc\'ia-Pacheco, Palmberg and Seoane-Sep\'ulveda  showed that the set of bounded functions with removable singularities in each rational, and everywhere else continuous and zero-valued, is $\mathfrak{c}$-lineable and $\mathfrak{c}$-algebrable. In particular, this shows that the set of Riemann integrable, nowhere Newton integrable\footnote{Recall that we say that a function $F: [a,b] \rightarrow \mathbb{R}  $ is the \emph{antiderivative} of $f: [a,b]\rightarrow \mathbb{R} $ if $F$ is continuous and $F' = f$ in $(a,b)$. Some authors refer to functions which have antiderivatives as \emph{Newton integrable}, and this terminology is naturally convenient for us.} functions is $\mathfrak{c}$-lineable and $\mathfrak{c}$-algebrable. This comes from the fact that bounded, a.e. continuous functions are all Riemann integrable, and derivatives never have removable discontinuities. 

In this section we investigate the set, which we will denote by $\mathcal{F}$, of functions of bounded variation with a dense set of jump discontinuities. This set is related to the mentioned result of \cite{pathos}, since each element of $\mathcal{F}$ is also Riemann integrable and nowhere Newton integrable; only this time the nowhere Newton integrability comes from the fact that derivatives do not have \emph{jump} discontinuities. Let us denote by $BV[0,1]$ the set of real-valued, left-continuous functions of bounded variation, equipped with the norm 
$$
\|f\| \doteq |f(0)| + V(f),
$$
where $V(f)$ is the total variation of $f$. Recall that, equipped with this norm, $BV[0,1]$ is a non-separable Banach space. Fix   an enumeration $\{q_i: i\in\mathbb{N}\}$  of all rational numbers in $(0,1)$, and consider
\begin{eqnarray}
f=\sum_{n=1}^\infty 2^{-i}\chi_{(q_i,1]}. 
\label{ourf}
\end{eqnarray}
Then $f$ belongs to $\mathcal{F}$; more specifically, it is monotone, continuous in $[0,1]\setminus \mathbb{Q}$, and has a jump discontinuity in each rational number in $(0,1)$. Constructions based on this function will be used to prove the main results of this section, namely Theorems \ref{nonsep} and \ref{Falg}. 

\begin{theorem}

There is a non-separable closed subspace of $BV[0,1]$ contained in $\mathcal{F} \cup \{0\}$. 

\label{nonsep}
\end{theorem}

\begin{proof} Define the Vitalli relation on $(0,1)$ in the following way: $x\sim y\iff x-y\in\mathbb{Q}  $. Let $V$ be a selector of the Vitalli relation, consider the function $f$ defined in (\ref{ourf}), and for each $v\in V$ define
$$
f_v(x)\doteq
\left\{\begin{array}{ll}
f(x-v), & \textrm{ if }x-v\in[0,1],\\
f(x-v+1), & \textrm{ otherwise}. 
\end{array} \right.
$$ 
Note that the one-sided limits of $f_v$ at $v$ are $f_v(v^-)=1$ and $f_v(v^+)=0$.  

Take any distinct $v,v'\in V$. Then $f_{v'}$ is continuous at $v$, $(f_{v'}-f_v)(v^-)=f_{v'}(v)-1$ and $(f_{v'}-f_v)(v^+)=f_{v'}(v)$, from which follows that
$$
V(f_{v'}-f_v)\geq 1.
$$
This shows that the set $\{f_v:v\in V\}$ is a discrete subset of $BV[0,1]$ of cardinality $\mathfrak{c}$. Let $f=\alpha_1 f_{v_1}+...+\alpha_k f_{v_k}$ be a linear combination such that $\alpha_i\neq 0$ and $v_i$ are distinct. Then the set of all points of jump discontinuity of $f$ equals $(0,1)\cap(\{v_i:i=1,\dots,k\}+ \mathbb{Q} )$.  

Now, let $g\in\overline{span\{f_v:v\in V\}}$. Then there is a sequence $(v_i)$ of distinct elements of $V$ such that $g\in \overline{span\{f_{v_i}:i\in\mathbb{N} \}}$. Let us show that $(f_{v_i})$ is a basic sequence. Note that $\|f_{v_i}\|\leq 3$. Fix a sequence of real numbers $(\beta_i)_i$, and $n < m$. Then
$$
\|\beta_1f_{v_1}+\cdots+\beta_nf_{v_n}\|\leq\sum_{i=1}^n|\beta_i|\|f_{v_i}\|=
3\sum_{i=1}^n|\beta_i|.
$$ 
For simplicity put $h=\beta_1f_{v_1}+\cdots+\beta_mf_{v_m}$. Then $|h(v_i^+) - h(v_i^-)|=|\beta_i|$, and therefore
$$
\|h\|\geq V(h)\geq \sum_{i=1}^m|\beta_i|.
$$
Hence,
$$
\|\beta_1f_{v_1}+\cdots+\beta_nf_{v_n}\|\leq 3\|\beta_1f_{v_1}+\cdots+\beta_mf_{v_m}\|,  
$$
and $(f_{v_i})_i$ is a basic sequence. Then there exists a sequence $(a_i)$ of real numbers with $g=\sum_{i=1}^\infty a_if_{v_i}$. Clearly the set of discontinuities of $g$ is dense in $(0,1)$. \end{proof}

It is possible to prove the spaceability of $\mathcal{F}$ more directly, and without using non-measurable sets, as a consequence of Theorem \ref{Gspaceable}. Non-separability, although, is lost. We will include this alternative proof at the end of the section.

\begin{theorem}
$\mathcal{F}$ is strongly $\mathfrak{c}$-algebrable. 
\label{Falg}
\end{theorem}

To prove Theorem \ref{Falg}, we will need Lemmas \ref{leminha} and \ref{lemao}.

\begin{lemma}

The algebra $\mathcal{A}$ of continuous functions in $[0,1]$ which are analytic in $(0,1)$ admits a set of free generators of cardinality $\mathfrak{c}$. 

\label{leminha}
\end{lemma}.

\begin{proof} Take $\mathbb{Q} $-linearly independent set of real numbers $\{r_\alpha: \alpha <\mathfrak{c}\}$. Then $\{x\mapsto e^{r_\alpha x}: \alpha <\mathfrak{c}\}$ is a set of free generators. \end{proof}

\begin{lemma}

Suppose that $G_j\in \mathcal{A}$, $j=1,\dots,k$, and that $G_k$ is non-zero. Then 
$$
g \doteq \sum_{j=1}^k G_j f^j, 
$$
where $f$ was defined in (\ref{ourf}), satisfies the following properties: 
\begin{enumerate}
\item both one-sided limits of $g$ exist at each $x\in (0,1)$, and for each $i\in\mathbb{N}  $ these limits at $q_i$ are given by $g(q_i^+)=\sum_{j=1}^kG_j(q_i)(f(q_i)+2^{-i})^j$ and $g(q_i^-)=\sum_{j=1}^k G_j(q_i)f^j(q_i)$; 
\item $g$ is in $BV[0,1]$, and in particular it is left-continuous; 
\item $g$ has has a dense set of jump discontinuities. 

\end{enumerate} 
\label{lemao}
\end{lemma}

\begin{proof} (1) and (2) are easily checked. Let us prove (3). Suppose first that $g=G_1 f$, where $G_1\in\mathcal{A}\setminus\{0\}$. For any subinterval $[a,b]\subset [0,1]$, since $G_1$ is analytic in $(0,1)$ and non-zero, there is an $x_0 \in [a,b]$ such that $G_1(x_0) \neq 0$. By the continuity of $G_1$, there is a rational number $q\in (a,b)$ satisfying $G_1(q) \neq 0$. Then
$$
(G_1 f)(q^-) = G_1(q) f(q^-) \neq G_1(q) f(q^+) = (G_1 f)(q^+), 
$$
thus $g=G_1 f$ has a jump discontinuity at $q$. 

Suppose now that the Lemma holds for $1,\dots,k-1$, and let $g= \sum_{j=1}^k G_j f^j$, where $G_1,\dots,G_{k}\in\mathcal{A}$ and $G_k$ is non-zero. For clearness, denote $a_i\doteq 2^{-i}$. For each $i\in\mathbb{N}$ we have that
\begin{align*}
g(q_i^+)-g(q_i^-)&=\sum_{j=1}^k G_j(q_i) ((f(q_i)+a_i)^j-f^j(q_i)) \\
&= \sum_{j=1}^k G_j(q_i) \left(\sum_{m=0}^j\left(j\atop m\right)a_i^mf^{j-m}(q_i)-f^j(q_i)\right)\\
&=\sum_{j=1}^k\ G_j(q_i)\sum_{m=1}^j\left(j\atop m\right)a_i^mf^{j-m}(q_i)\\
&=\sum_{m=1}^k a_i^m P_m(G_m(q_i),\dots,G_k(q_i),f(q_i)), 
\end{align*}
where $P_m(x_m,\dots,x_k,y)\doteq\sum_{j=m}^k\left(j\atop m\right)x_j y^{j-m}$. Note that there is an $M>0$ such that 
$$
|G_j(x)| \leq M,
$$ 
for all $j=1,\dots,k$ and all $x\in [0,1]$. Since $P_m$ are polynomials, there is an $N>0$ such that 
$$|P_m(x_m,\dots,x_k,y)|\leq N,$$
for all $m=1,\dots,k$ and all $(x_m,\dots,x_k,y)\in [-M,M]^{k-m}\times [0,1]$. In particular, 
$$
|P_m(G_m(x),\dots,G_k(x),f(x))|\leq N$$ 
for all $m=1,\dots,k$ and $x\in [0,1]$. 

Let $[a,b]$ be any subinterval of $(0,1)$. Consider the function 
$$
h(x) \doteq P_1(G_1(x),\dots,G_k(x),f(x))
= G_1(x) + \sum_{j=2}^k \left(j\atop 1\right) G_j(x) f^{j-1}(x).   
$$ 
By the induction hypothesis, $\sum_{j=2}^k \left(j\atop 1\right) G_j(\cdot) f^{j-1}(\cdot)$ has has a dense set of jump discontinuities in $[a,b]$. Since $G_1$ is continuous, $h$ has jump discontinuities at the same points, and in particular it does not vanish at $(a,b]$. Let $\varepsilon >0$ and $x_0\in (a,b]$ be such that $|h(x_0)| > \varepsilon$. Since $h$ is left-continuous, $|h(x)| > \varepsilon$ for each $x < x_0$ which is close enough to $x_0$. In particular, the set 
$$
S\doteq \{i:q_i\in [a,b], |h(q_i)|\geq \varepsilon\}
$$ 
is infinite. For each $i\in S$ we have that
\begin{align*}
|g(q_i^+)-g(q_i^-)|&=\left|\sum_{m=1}^k a_i^m P_m(G_m(q_i),\dots,G_k(q_i),f(q_i))\right|\\
&\geq|a_i h(q_i)|- \sum_{m=2}^k|a_i^m P_m(G_m(q_i),\dots,G_k(q_i),f(q_i))|\\
&\geq a_i\varepsilon-\sum_{m=2}^k a_i^mN= \frac{1}{2^i}\varepsilon-N\sum_{m=2}^k\frac{1}{2^{mi}}\geq
\frac{1}{2^i}\varepsilon-N\sum_{m=2}^\infty\frac{1}{2^{mi}}\\
&=\frac{1}{2^i}\left(\varepsilon-N\frac{1}{2^i-1}\right). 
\end{align*}
Since $S$ is infinite, there is $i_0\in S$ such that $\varepsilon-N\frac{1}{2^{i_0}-1} > 0$. Then $g$ has a jump discontinuity at $q_{i_0}$.  \end{proof}

\begin{proof}[Proof of Theorem \ref{Falg}]
 Let $\{g_\alpha:\alpha <\mathfrak{c}\}$ be a set of free generators in $\mathcal{A}$, and consider the function $f$ defined in (\ref{ourf}). We will show that $\{g_\alpha f:\alpha <\mathfrak{c}\}$ is a set of free generators in $\mathcal{F}$, and that the algebra generated by this set is entirely contained in $\mathcal{F}\cap\{0\}$. Let $p$ be a polynomial of $n$ variables and real coefficients, with no constant term, given by 
$$
p(x_1,\dots,x_n) \doteq  \sum_{\gamma\in\Gamma} \lambda_\gamma x_1^{\gamma_1}\ldots x_n^{\gamma_n};
$$
here, $\Gamma$ is denoting a (finite) set of $n$-dimensional multi-indices of the form $\gamma = (\gamma_1,\dots,\gamma_n)$, and each $\lambda_\gamma$ is nonzero. Note that $(0,\dots,0)\not\in \Gamma$. To complete our proof, it suffices to show that $g\doteq p(g_{\alpha_1}f,\ldots,g_{\alpha_n}f)$ is in $\mathcal{F}\cup \{0\}$ and that $g$ is zero if and only if $\lambda_\gamma = 0$ for all $\gamma\in\Gamma$. We can write $g$ as follows: 
$$
g=
\sum_{\gamma\in\Gamma} \lambda_\gamma g_{\alpha_1}^{\gamma_1}\ldots g_{\alpha_n}^{\gamma_n}f^{|\gamma |} = 
\sum_{j=1}^k \left[ \sum_{\gamma\in\Gamma,\,|\gamma |=j}\lambda_\gamma g_{\alpha_1}^{\gamma_1}\ldots g_{\alpha_n}^{\gamma_n} \right] f^j
=\sum_{j=1}^k G_j f^j,  
$$
where each $G_j \doteq \lambda_\gamma g_{\alpha_1}^{\gamma_1}\ldots g_{\alpha_n}^{\gamma_n}$ clearly belongs to $\mathcal{A}$. Note that, since $\{g_\alpha:\alpha <\mathfrak{c}\}$ is a set of free generators in $\mathcal{A}$, then for each $j$, $G_j = 0$ implies that $\lambda_\gamma = 0$ for each $\gamma\in\Gamma$ satisfying $|\gamma|=j$. The result follows from Lemma \ref{lemao}.
\end{proof}

\subsection{An alternative proof of the spaceability of $\mathcal{F}$.} Recall that $L^1$ is isometrically embedded in $BV[0,1]$ via the natural map $\psi$ which takes each Lebesgue integrable function $g$ into its primitive $G(x)\doteq \int_0^x g$. The image of $\psi$ is denoted by $AC_0$, and it consists of all absolutely continuous functions starting at zero. Note that each nowhere essentially bounded Lebesgue integrable function is mapped via $\psi$ into a continuous function which is not identically zero in any subinterval of the domain; therefore by Theorem \ref{Gspaceable} there is an infinite dimensional closed subspace $S$ of $BV[0,1]$, each non-zero element of which satisfies that property. Note also that, for each non-zero $g\in S$, $(f+1)g$ has dense set of jump discontinuities, where $f$ was defined in (\ref{ourf}). The conclusion follow from the fact that  $g \in AC_0 \stackrel{\varphi}{\mapsto} (f+1)g \in BV[0,1]$ is an isomorphism onto its image, and thus $\varphi[S]$ is a closed, infinite-dimensional subspace of $BV[0,1]$ contained in $\mathcal{F}\cup\{0\}$.

%%%%%%%%%%%%%%%%%%%%%%%%%%%%%%%%%%%%%%%%%%%%%%%%%%%%%%%%%%%%%%%%%%%%%%%%%%

\section{Spaceability of sets of Kurzweil integrable functions}
\label{seckurnleb}

Consider the vector space $K([0,1])$ of all Kurzweil integrable functions equipped with the Alexievicz norm given by 
$$\|f\|_A \doteq \|F\|_\infty,$$
where $F$ is the (Kurzweil) primitive function of $f$. Like in $L^1$, in $K([0,1])$ equivalence classes of almost everywhere equal functions are considered, and Kurzeil integrable functions which are equal a.e. have the same primitive. Since all primitives to Kurzweil integrable functions are continuous, $K([0,1])$ can be isometrically imbedded, through the natural identification $f\mapsto F$, onto a (dense, since all polynomials are primitives for Kurzweil integrable functions) subspace of $C([0,1])$, which we will denote by $\mathbb{K}$, or $\mathbb{K}([0,1])$. One of the main properties of $\mathbb{K}$ is that it includes $AC_0$ (which is another way of saying that all Lebesgue integrable functions are Kurzweil integrable) and $C^0([0,1])$, the subspace of $C([0,1])$ consisting of all differentiable functions; for details see e.g. \cite{gor}. Another feature of the Kurzweil integral, lacked by the Lebesgue integral, is Hake's Theorem: 

\begin{theorem}[Hake]
\label{teohake}
$f:[0,1]\rightarrow \mathbb{R}$ is Kurzweil integrable if and only if for each small $\epsilon>0$, the restriction $f|_{[\epsilon,1]}\in K([\epsilon,1])$, and the limit 
\begin{eqnarray}
\lim_{\epsilon\rightarrow 0} \int_{[\epsilon,1]} f
\label{eqhake}
\end{eqnarray}
exists. In that case, $\int_{[0,1]} f$ equals (\ref{eqhake}). 

\end{theorem}

We will also need the following standard result which relates the Kurzweil and the Lebesgue integrals:

\begin{proposition}

$f:[0,1]\rightarrow \mathbb{R}$ is Lebesgue integrable if and only if $f$ and $|f|$ are Kurzweil integrable, and in that case the values of the Lebesgue and the Kurzweil integrals of $f$ coincide. 

\label{|f|}
\end{proposition}

We are interested in the set $\mathcal{J}$ of Kurzweil integrable functions which are not Lebesgue integrable. In the following we will construct a class of such functions which satisfy nice properties. Define $\Phi : [0,1] \rightarrow \mathbb{R}$ by
$$\Phi (x)\doteq \left\{ 
\begin{array}{ll}
4 x^2 \sin \left(\frac{\pi}{4 x^2}\right) & \mbox{if } x\in (0,1/2];\\
-4 (1-x)^2 \sin \left(\frac{\pi}{4 (1-x)^2}\right) & \mbox{if } x\in (1/2,1);\\
0 & \mbox{if } x = 0 \mbox{ or } x=1.
\end{array} \right. $$
Then $\Phi$ is everywhere differentiable and  
$$\phi(x) \doteq \Phi '(x)= \left\{ 
\begin{array}{ll}
8x \sin \left(\frac{\pi}{4 x^2}\right) - \frac{2\pi}{x} \cos \left(\frac{\pi}{4 x^2}\right) & \mbox{if } x\in (0,1/2];\\
8(1-x) \sin \left(\frac{\pi}{4 (1-x)^2}\right) - \frac{2\pi}{1-x} \cos \left(\frac{\pi}{4 (1-x)^2}\right) & \mbox{if } x\in (1/2,1);\\
0 & \mbox{if } x\in \mathbb{R} \setminus (0,1).
\end{array} \right. $$
For each subinterval $I=[a,b]$, define $\Phi_I : [0,1] \rightarrow \mathbb{R}$ by
\begin{eqnarray}
\Phi_I(x) \doteq 
\left\{ 
\begin{array}{ll}
\Phi \left(\frac{x-a}{b-a}\right) & \mbox{if } x\in I;\\
0 & \mbox{otherwise.}
\end{array} \right.
\label{Phi}
\end{eqnarray}
Then $\Phi_I$ is also everywhere differentiable, and the derivative is given by 
\begin{eqnarray}
\phi_I(x) \doteq \Phi_I'(x)  =
\left\{ 
\begin{array}{ll}
\frac{1}{b-a}\, \phi \left(\frac{x-a}{b-a}\right) & \mbox{if } x\in I;\\
0 & \mbox{otherwise.}
\end{array} \right.
\label{phi}
\end{eqnarray}
Each $\phi_I$ is Kurzweil integrable (with ${(K)\!\!\int} \phi_I = 0$), but is not Lebesgue integrable, since its positive and negative parts have infinite integral. For each natural number $k$, consider $I_k\doteq [2^{k+1},2^k]$. Clearly $\{\phi_{I_k}:k\in \mathbb{N}\}$ is linearly independent. We can now state our first result on lineability for sets of Kurzweil integrable functions:

\begin{proposition}

$span(\{\Phi_{I_k}:k\in \mathbb{N}\}) \subset \left( C^0([0,1]) \setminus AC_0\right)\cup\{0\}$. In particular, $\left( C^0([0,1]) \setminus AC_0\right)$ and $\mathcal{J}$ are lineable. 

\label{KLlin}
\end{proposition} 

\begin{proof} It suffices to show that $span(\{\phi_{I_k}:k\in \mathbb{N}\})\cap L^1([0,1]) = \{0\}$. Let $f$ be a nontrivial linear combination of the $\phi_{I_k}$. Then there is a nonzero sequence $(\alpha_k)_k\in c_{00}$ such that
$$f = \sum_k \alpha_k\phi_{I_k}.$$ 
Suppose that $\alpha_{k_0}\neq 0$. Then $f|_{I_{k_0}} = \alpha_{k_0}\phi_{I_{k_0}}$, which is not Lebesgue integrable. 
\end{proof}

\begin{remark}
It is natural to ask wether Proposition \ref{KLlin} can be generalized analogously to Corollary \ref{corolGspaceable} or not. That is, whether the set of Kurzweil integrable functions which are not Lebesgue integrable in \emph{any} subinterval is lineable or not. The generalization is impossible, the reason being that, as it is known (see again \cite{gor}), each Kurzweil integrable function must be Lebesgue integrable in some subinterval.  
\end{remark}

\begin{proposition}
\label{propKspac}
$\overline{span(\{\phi_{I_k}:k\in \mathbb{N}\})}^{\|\cdot\|_A} \subset \mathcal{J}\cup\{0\}$. In particular, $\mathcal{J}$ is spaceable in $K([0,1])$. 

\end{proposition} 

\begin{proof} Suppose that $f\in \overline{span(\{\phi_{I_k}:k\in \mathbb{N}\})}^{\|\cdot\|_A}$. Consider a sequence $(f^m)_m$, such that each $f^m \in span(\{\phi_{I_k}:k\in \mathbb{N}\})$. Then for each $m$ there is a sequence $(\alpha_k^m)\in c_{00}$ such that 
$$f^m = \sum_k \alpha_k^m \phi_{I_k}.$$
An easy computation gives us that, for each $k$, $(\alpha_k^m)_m$ converges in $m$ to a real number $\alpha_k$, and that we can write
\begin{eqnarray}
f = \sum_k \alpha_k \phi_{I_k}.
\label{eqf}
\end{eqnarray}
An argument identical to the one in the proof of Proposition \ref{KLlin} shows that $f$ is not Lebesgue integrable. \end{proof}

If we consider, instead of the closure in $K([0,1])$ of $span(\{\phi_{I_k}:k\in \mathbb{N}\})$, the closure in $C([0,1])$ of $span(\{\Phi_{I_k}:k\in \mathbb{N}\})$ (with the sup norm), then it is easily seen that a function $F \in \overline{span(\{\Phi_{I_k}:k\in \mathbb{N}\})}^{\|\cdot\|_\infty}$, as in (\ref{eqf}), is of the form 
$$F = \sum_k \alpha_k \Phi_{I_k}.$$
For each small $\epsilon >0$, $F|_{[\epsilon,1]}\in \mathbb{K}([\epsilon,1])$, and by the continuity of $F$ the limit $\lim_{\epsilon\rightarrow 0}(F(1) - F(\epsilon))$ exists. Then by Hake's Theorem \ref{teohake}, $F\in \mathbb{K}([0,1])$. Thus, in addition to Proposition \ref{propKspac}, we have the following:

\begin{proposition}
\label{gura}
Given $x_0\in [0,1]$, $C^0([0,1])$ admits an infinite dimensional subspace $S$ such that each element $F$ of $\overline{S}^{\|\cdot\|_\infty}$ in $C([0,1])$ satisfies: 
\begin{enumerate}
\item $F$ is a primitive to a Kurzweil integrable function; 
\item $F$ is differentiable everywhere except perhaps at $x_0$.

\end{enumerate}  

\end{proposition}

\begin{remark} As it was mentioned, Gurariy proved that $C^0([0,1])$ is not spaceable in $C([0,1])$, which means that every infinite dimensional space of differentiable functions has limit points which are not differentiable.  Proposition \ref{gura} allows us to expect some control on the closure of infinite dimensional spaces of differentiable functions, if these spaces are constructed carefully. This raise the natural question: which are the \emph{minimal} conditions that a closed subspace of $C([0,1])$, containing an infinite dimensional subspace of $C^0([0,1])$, must have?  In \cite{bpp}, Bongiorno, Di Piazza and Preiss have given a constructive definition of integral, called \emph{C-integral}, whose set of primitives is exactly $AC+C^0([0,1])$. Motivated by this we might ask if there is a closed subspace $E$ of $C([0,1])$, containing an infinite dimensional subspace of $C^0([0,1])$, and such that $E\subset AC+C^0([0,1])$.  Since there is not a Hake-like Theorem for the C-integral, we are not able to answer this question with an argument analogue to the one used to prove Proposition \ref{gura}. 
\end{remark}

\subsection{Remark on the algebrability of $\mathcal{J}$}

Unlike the sets studied in the previous sections (see Theorems \ref{lnebalg} and \ref{Falg}), the set $\mathcal{J}\cup \{0\}$ does not admit algebraic structures within it at all. This is a direct consequence of the following fact:  if $g$ is a Kurzweil integrable function such that $g^2$ is Kurzweil integrable, then $g^2$ is also Lebesgue integrable by Proposition \ref{|f|}.

%%%%%%%%%%%%%%%%%%%%%%%%%%%%%%%%%%%%%%%%%%%%%%%%%%%%%%%%%%%%%%%%%%%%%%%%

\bibliographystyle{amsplain}

\begin{thebibliography}{99}

\bibitem{ags} R. Aron, V. I. Gurariy, J. B. Seoane, 
\textit{Lineability and spaceability of set of functions on $\mathbb{R}$,}
Proc. Amer. Math. Soc., \textbf{133} (2005), 795--803.  

\bibitem{as} R. Aron, J. B. Seoane-Sep\'ulveda,
\textit{Algebrability of the set of everywhere surjective functions on $\mathbb{C}$,} 
Bull. Belg. Math. Soc. Simon Stevin, \textbf{14} (2007), 25--31. 



\bibitem{bsz} A. Bartoszewicz and Sz. G\l \c ab, 
\textit{Strong algebrability of sets of sequences and functions,}
Proc. Amer. Math. Soc. (to appear). 

\bibitem{bpp} B. Bongiorno, L. Di Piazza, and D. Preiss, 
\textit{A constructive minimal integral which includes Lebesgue integrable functions and derivatives,}
J. London Math. Soc., \textbf{62} (2000), no. 1, 117--126. 

\bibitem{eg} P. Enflo,  V. I. Gurariy and  J. B. Seoane-Sepúlevda
\textit{Some results and open questions on spaceability in function spaces,} 
Trans. Amer. Math. Soc.  (to appear). 

%\bibitem{eg} P. Enflo and V. I. Gurariy, 
%\textit{On lineability and spaceability of sets in function spaces,}
%unpublished notes. 



\bibitem{gms} F. J. Garc\'{i}a-Pacheco, M. Mart\'{i}n, and J. B. Seoane-Sep\'ulveda. \textit{Lineability,
spaceability, and algebrability of certain subsets of function spaces,} Taiwanese
J. Math., \textbf{13} (2009), no. 4, 1257--1269.


\bibitem{pathos} F. J. Garc\'ia-Pacheco, N. Palmberg and J. B. Seoane-Sep\'ulveda, 
\textit{Lineability and algebrability of pathological phenomena in analysis,} 
J. Math. Anal. Appl., \textbf{326} (2007), 929--939. 


\bibitem{gor} R. A. Gordon, 
\textit{The Integrals of Lebesgue, Denjoy, Perron and Henstock,}
Amer. Math. Soc., Providence, RI, 1994. 

\bibitem{gurrus}  V. I.  Gurariy, 
\textit{Subspaces and bases in spaces of continuous functions (Russian),}
Dokl. Akad. Nauk SSSR, \textbf{167} (1966), 971--973. 

\bibitem{gur}  V. I. Gurariy,
\textit{Linear spaces composed of nondifferentiable functions,}
C. R. Acad. Bulgare Sci., \textbf{44} (1991), 4 13--16. 



\end{thebibliography}

\end{document}